\documentclass[draft]{amsart}

\usepackage{amsthm}
\usepackage{amssymb}
\usepackage{color}

\newtheorem{thma}{Theorem} 
\newtheorem{thmaa}{Theorem} 

\newtheorem{thm}{Theorem}[section]
\newtheorem{cor}[thm]{Corollary}

\newtheorem{prop}[thm]{Proposition}

\theoremstyle{definition}
\newtheorem{defn}[thm]{Definition}

\theoremstyle{remark}

\newcommand{\Aut}{\operatorname{Aut}}
\newcommand{\Out}{\operatorname{Out}}

\newcommand{\diam}{\operatorname{diam}}

\newcommand{\cv}{\operatorname{cv}}

\newcommand{\rk}{\operatorname{rk}}
\newcommand{\ncl}{\operatorname{ncl}}

\newcommand{\Cay}{\operatorname{Cay}}

\newcommand{\Cyl}{\operatorname{Cyl}}
\newcommand{\Curr}{\operatorname{Curr}}

\numberwithin{equation}{subsection}

\title{spectral rigidity and subgroups of free groups}
\author{Brian Ray}
\address{Department of Mathematics, University of Illinois at Urbana-Champaign, 1409 West Green Street, Urbana, IL 61801, USA}
\email{ray8@illinois.edu}
\date{}
\thanks{The author acknowledges support from National Science Foundation grant DMS-0904200.}

\begin{document}

\begin{abstract}
A subset $\Sigma \subset F_N$ of the free group of rank $N$ is called \emph{spectrally rigid} if whenever trees $T, T'$ in Culler-Vogtmann Outer Space are such that $\| g \|_T = \| g \|_{T'}$ for every $g \in \Sigma$, it follows that $T = T'$.  Results of Smillie, Vogtmann, Cohen, Lustig, and Steiner prove that (for $N \geq 2$) no finite subset of $F_N$ is spectrally rigid in $F_N$.  We prove that if $\{ H_i \}_{i=1}^k$ is a finite collection of subgroups, each of infinite index, and $g_i \in F_N$, then $\cup_{i=1}^k g_i H_i$ is not spectrally rigid in $F_N$.  Taking $H_i = 1$, we recover the results about finite sets.  We also prove that any coset of a nontrivial normal subgroup $H \lhd F_N$ is spectrally rigid.
\end{abstract}

\maketitle

\section{Introduction}

The Culler-Vogtmann Outer Space (denoted $\cv_N$) was introduced by Culler and Vogtmann \cite{MR830040} as a free group analog of Teichmuller Space.  While the latter admits an action by the Teichmuller Modular Group, Outer Space admits an action by $\Out(F_N)$, the outer automorphism group of $F_N$.  The space $\cv_N$ has proven to be a key tool in studying $\Out(F_N)$.

Formally, $\cv_N$ is the space of all free, minimal, discrete, isometric actions of $F_N$ on $\mathbb{R}$-trees, $T$, considered up to $F_N$-equivariant isometry.  One may think of a point in the space as a triple $(\Gamma,\tau,l)$ where $\Gamma$ is a finite graph equipped with both a marking isomorphism, $\tau \colon F_N \to \pi_1(\Gamma)$, and a function $l \colon E(\Gamma) \to \mathbb{R}_{> 0}$ which assigns a length to each edge.  Given such a triple, one sets $T = \widetilde{\Gamma}$ so that $T$ is equipped with a covering space action of $F_N \cong \pi_1(\Gamma)$ which is free, minimal, discrete, and isometric.  One obtains $(\Gamma,\tau,l)$ from $T$ by considering the quotient graph $\Gamma = F_N \backslash T$.  To each $T$ we associate a length function $\| \cdot \|_T \colon F_N \to \mathbb{R}_{\geq 0}$ defined by $\| g \|_T = \inf_{x \in T} \{ d(x,gx) \}$.  Equivalently, $\| g \|_T$ is the length of the cyclically reduced loop (in $\Gamma$) which represents the conjugacy class of $\tau(g)$.  The closure, $\overline{\cv}_N$, of $\cv_N$ is known to consist of the so called very small actions of $F_N$ on $\mathbb{R}$-trees \cite{BF,MR1341810}.  The fact that $\| \cdot \|_T$ determines $T$ \cite{MR1851337} gives an injection $\overline{\ell} \colon \overline{\cv}_N \to (\mathbb{R}_{\geq 0})^{F_N}$; we call $\overline{\ell}(T)$ the \emph{marked length spectrum} of $T$.  Let $\ell = \overline{\ell}|_{\cv_N}$ and note that $\ell$ is also injective.  Morally speaking, a subset $\Sigma \subset F_N$ is \emph{spectrally rigid} (resp. \emph{strongly spectrally rigid}) if one can replace the target of $\ell$ (resp. $\overline{\ell}$) by $(\mathbb{R}_{\geq 0})^{\Sigma}$ and still retain injectivity.  The term was recently coined (in the free group setting) by Kapovich \cite{MR2869139}, though similarly spirited work dates back to the early 1990's (see \cite{MR1105334,MR1182503}).

\begin{defn}[(Strongly) Spectrally rigid]\label{sr}
We say $\Sigma \subseteq F_N$ is (\emph{strongly}) \emph{spectrally rigid} if whenever $T_1, T_2 \in \cv_N$ (resp. $\overline{\cv}_N$) are such that $\|g\|_{T_1} = \|g\|_{T_2}$ for every $g \in \Sigma$, then $T_1 = T_2$ in $\cv_N$ (resp $\overline{\cv}_N$).
\end{defn}

For $N \geq 3$ Smillie and Vogtmann \cite{MR1182503} prove that no finite set is spectrally rigid in $F_N$.  More specifically, they show that given a finite set, $\Sigma$, of conjugacy classes there exists a one parameter family of trees, $T_t \in \cv_N$, whose translation length functions all agree on $\Sigma$.  For $N = 2$, Cohen, Lustig and Steiner \cite{MR1105334} provide a similar argument as to why no finite set is spectrally rigid in $F_N$.  We prove the following.

\begin{thma}\label{A}
Let $N \geq 2$.  Let $\{ H_i \}_{i=1}^k$ be a finite collection of finitely generated subgroups $H_i < F_N$.  Let $\{ g_i \}_{i=1}^k$ be a finite collection of elements $g_i \in F_N$. Let $\mathcal{H} = \cup_{i=1}^k g_iH_i$.  Then the following are equivalent.
\begin{enumerate}
\item For every $i$, $[F_N \colon H_i] = \infty$.
\item $\mathcal{H} \subset F_N$ is not spectrally rigid in $F_N$.
\end{enumerate}
\end{thma}

In the above theorem, we recover the results of Smillie, Vogtmann, Cohen, Lustig, and Steiner by setting $H_i = \{1\}$.  The proof by Smillie and Vogtmann uses the fact that in a finite set of conjugacy classes, there is a universal bound on the exponent with which a given primitive element may occur.  Their argument, however, works for \emph{any} set $\Sigma$ which satisfies the following property: there exists a triple $(\mathcal{A},a,M)$ with $\mathcal{A}$ a free basis, $a \in \mathcal{A}$, and $M \geq 1$, so that for any $g \in \Sigma$, if $a^k$ occurs as a subword of the cyclically reduced form (over $\mathcal{A}$) of $g$, then $|k| \leq M$ (see Definition \ref{W} and Theorem \ref{CLSRSV}).  To prove Theorem \ref{A}, we use \emph{laminations associated to a fully irreducible automorphism}.  Introduced by Bestvina, Feighn, and Handel in \cite{MR1445386} these laminations are defined -- for a given fully irreducible $\varphi \in \Out(F_N)$ -- in terms of a train track representative $f \colon \Gamma \to \Gamma$.  Leaves of the lamination are ``generated'' by edgepaths of the form $f^n(e)$ for any $e \in E(\Gamma)$ and $n \in \mathbb{N}$ (see \cite{MR1445386,2011arXiv1104.1265K}).  It is a theorem of Bestvina, Feighn, and Handel \cite{MR1445386} that only finitely generated subgroups of finite index may ``carry a leaf'' of such a lamination.  It follows that, for finitely generated infinite index subgroups, edgepaths of the form $f^n(e)$ (for suitably large $n$) can not be ``read'' in $H$.  Thus if a primitive loop contains such an $f^n(e)$, then it cannot be a subword of any $h \in H$ (compare Definition \ref{W}).

Theorem \ref{A} fails if we allow infinite unions.  Indeed, consider $\mathcal{H} = \cup_{g \in F_N} g \{ 1 \}$.  Then $\mathcal{H} = F_N$ and so is spectrally rigid.  Since translation length functions satisfy $\| g^n \| = n \| g \|$, any finite index subgroup is spectrally rigid.  More generally, if $\pi \colon F_N \to G$ is a presentation homomorphism ($G \cong F_N / \ker \pi$) for a torsion group, $G$, then $\ker \pi$ is spectrally rigid.  Note that this applies to (normal) subgroups of infinite index (in the case $G$ is infinite), and these subgroups are necessarily infinitely generated.  Thus the finite generation assumption in Theorem \ref{A} is also essential.  

Carette, Francaviglia, Kapovich, and Martino prove \cite{MR2966693} that if $H < \Aut(F_N)$ is such that the image of $H$ in $\Out(F_N)$ contains an infinite normal subgroup, then for any $g \in F_N$ (so long as $N \neq 2$ or $g$ not conjugate to a power of $[a,b]$ in $F_2 = F(a,b)$) the orbit $Hg$ is spectrally rigid.  This shows, for example, that the commutator subgroup, $[F_N,F_N]$, of $F_N$ is spectrally rigid, as it contains the orbit $\Aut(F_N)[g,h]$ for a suitable commutator $[g,h]$.  More generally, any nontrivial verbal or marginal subgroup is spectrally rigid since such groups are known to be characteristic \cite[\S 2.3]{MR1357169}.

Motivated by the above result regarding characteristic subgroups, Ilya Kapovich asked the following question: Is it true that for any nontrivial normal subgroup $H \lhd F_N$, $H$ is a spectrally rigid set?  We prove the following.

\begin{thma}\label{C}
Let $H \lhd F_N$ be a nontrivial normal subgroup.  Then for any $g \in F_N$, the coset $gH$ is strongly spectrally rigid.
\end{thma}

The proof uses \emph{geodesic currents} on free groups (see \cite{MR2216713,MR2496058,MR2585579,Martin}).  These are positive Radon measures on $\partial^2 F_N = \partial F_N \times \partial F_N \backslash \Delta$ (where $\Delta$ is the diagonal) which are invariant under the diagonal action of $F_N$ and under the map which exchanges the coordinates.  The space of all such currents is denoted $\Curr(F_N)$.  The spaces $\Curr(F_N)$ and $\overline{\cv}_N$ are related by an \emph{intersection form} \cite{MR2496058}, $\langle \cdot , \cdot \rangle \colon \overline{\cv}_N \times \Curr(F_N) \to \mathbb{R}_{\geq 0}$.  The main feature of the intersection form that we use is as follows: given a conjugacy class $g \in F_N$, there is a canonically associated current $\eta_g \in \Curr(F_N)$ such that for any $T \in \overline{\cv}_N$, we have $\langle T ,\eta_g \rangle = \| g \|_T$.  Now given any nontrivial $g, r \in F_N$, we show that the normal closure, $\ncl(r)$, of $r$ contains a sequence of words $w_n$ such that large powers of $g$ exhaust $w_n$.  Linearity of the intersection form then reveals that the associated counting currents converge to the counting current of $g$, and that if two trees $T, T'$ agree on $\ncl(r)$ (and hence on the $w_n$), then they must agree on $g$.  Since $g$ is arbitrary, we conclude $T = T'$.  Our result then follows from the fact that any nontrivial normal subgroup contains $\ncl(r)$ for some nontrivial $r$.


\section{Preliminaries}

The $N$-Rose, $R_N$, is the graph with one vertex and $N$ topological edges.  If $p$ is an edgepath in a graph, $G$, we let $[p]$ denote the reduced edgepath homotopic (rel. endpoints) to $p$.  If $p$ is a closed path, we let $[[p]]$ denote the cyclically reduced edgepath freely homotopic to $p$.  If $g \in F_N$ and $\mathcal{A}$ is a free basis, then $[g]_\mathcal{A}$ (resp $[[g]]_\mathcal{A}$) denotes the freely reduced (resp. cyclically reduced) form of $g$ over the basis $\mathcal{A}$.  If $v, w$ are edgepaths in $G$, we write $v \Subset w$ to indicate that $v$ is a subpath of $w$.  If $g,h \in F_N$ and a free basis, say $\mathcal{A}$, is specified, then $g \Subset h$ means the edgepath $g$ occurs as a subpath of the edgepath $h$ in the graph $G = R_\mathcal{A}$, the rose with petals labelled by the elements of $\mathcal{A}$.

\subsection{Property $\mathcal{W}$}

As mentioned in the introduction, the proof by Smillie and Vogtmann \cite{MR1182503} that no finite set is spectrally rigid in $F_N$ (for $N \geq 3$) admits a generalization.  Specifically, their argument proceeds verbatim so long as the subset $\Sigma \subset F_N$ satisfies the following property.

\begin{defn}[Property $\mathcal{W}$]\label{W}
Let $\Sigma \subseteq F_N$.  We say $\Sigma$ has \emph{property} $\mathcal{W}$ if there exist a free basis $\mathcal{A}$ of $F_N$, $a \in \mathcal{A}$, and $M \geq 1$ so that for any $\sigma \in \Sigma$, if $a^k \Subset [[\sigma]]_\mathcal{A}$ then $|k| \leq M$.
\end{defn}

In \cite{MR2922383} it is verified that the argument provided by Cohen, Lustig, and Steiner \cite{MR1105334} (for $N = 2$) also works under the assumption of property $\mathcal{W}$.

\begin{thm}[{CLS, R, SV\cite{MR1105334,MR2922383,MR1182503}}]\label{CLSRSV}
For $N \geq 2$, if a subset $\Sigma \subset F_N$ has Property $\mathcal{W}$, then $\Sigma$ is not spectrally rigid.
\end{thm}

\subsection{Stallings Subgroup Graphs}

Given a point $T = (\Gamma, \tau, l) \in \cv_N$, and a finitely generated subgroup $H \leq F_N$, we consider the (possibly infinite) cover of $\Gamma$ corresponding to the subgroup $\tau(H)$, denoted $(X^T_H)_r$.  Its core, denoted $X^T_H$, is finite (since $H$ is finitely generated) and represents the conjugacy class of $\tau(H)$ in $\pi_1(\Gamma)$.  Note that $X^T_H = (X^T_H)_r$ (i.e. $X^T_H$ is $E(\Gamma)$-regular) if and only if $[F_N : H] < \infty$ if and only if $X^T_H \to \Gamma$ is a covering.  If $H$ is of infinite index, then one obtains $(X^T_H)_r$ from $X^T_H$ by attaching ``hanging trees'' at vertices which are not $E(\Gamma)$-regular.  Note that varying the basepoint in $(X^T_H)_r$ corresponds to conjugation of $\tau(H)$ by elements of $\pi_1(\Gamma)$.  In the event the basepoint, $b$, in $(X^T_H)_r$ lies outside $X^T_H$, we let $(X^T_H)_b$ denote the graph $X^T_H \cup [X^T_H,b]$, where $[X^T_H,b]$ is the bridge between $X^T_H$ and $b$ in $(X^T_H)_r$.  For more information see, e.g., \cite{MR2396717, MR1882114, MR695906}.

\subsection{Train Track Maps}

By a \emph{graph map}, we mean a function $f \colon V(\Gamma) \cup E(\Gamma) \to V(\Gamma') \cup P(\Gamma')$ sending vertices to vertices and edges to edgepaths and for which the notions of incidence and inverse are preserved.  If $\mathcal{A}$ is a free basis of $F_N$ and $\varphi \in \Out(F_N)$, then we can always think of $\varphi$ as a graph map $R_N \to R_N$ by pairing $E(R_N)$ with $\mathcal{A}$.  By a \emph{marked graph} we mean a pair $(\Gamma,\tau)$, where $\Gamma$ is a graph and $\tau \colon R_N \to \Gamma$ is a homotopy equivalence.  In the event $\Gamma = R_N$, we write $R_N = R_\mathcal{A}$, with $\mathcal{A}$ the free basis determined by the pairing $E(R_N)$ with $\mathcal{A}$.  Given $\varphi \in \Out(F_N)$, and a marked graph $(\Gamma,\tau)$ we say that a graph map $f \colon \Gamma \to \Gamma$ is a \emph{topological representative} of $\varphi$ with respect to $\tau$ if $f$ is a homotopy equivalence and the outer automorphism determined by $\tau \circ f \circ \tau^{-1} \colon R_N \to R_N$ is equal to $\varphi \colon R_N \to R_N$.

\begin{defn}[Train track map]
A graph map $f \colon \Gamma \to \Gamma$ is called a \emph{train track map} if for every $e \in E(\Gamma)$, the edgepath $f^n(e)$ is not a vertex and is reduced.
\end{defn}

We say that $\varphi \in \Out(F_N)$ is \emph{reducible} if there is a free factorization $F_N = F^0 \ast \cdots \ast F^{l-1} \ast H$, with $l \geq1$ and $1 \leq \rk (F^0) < N$, so that $\varphi$ permutes the conjugacy classes of the $F^i$; we allow $H = 1$.  Otherwise, we say $\varphi$ is \emph{irreducible}.

\begin{thm}[{Bestvina and Handel \cite[Theorem 1.7]{MR1147956}}]\label{BH1.7}
Every irreducible automorphism is topologically represented by a train track map.
\end{thm}

If all positive powers of $\varphi$ are irreducible, then we say that $\varphi$ is \emph{irreducible with irreducible powers} (iwip for short).  Thus $\varphi$ is fully irreducible if and only if for all $k \geq 1$, $\varphi^k$ does not preserve the conjugacy class of a proper free factor.  In what follows, we use train track maps $f \colon \Gamma \to \Gamma$ representing iwip automorphisms in $\Out(F_N)$.  The existence of such automorphisms -- for every rank $N \geq 2$ -- is well known (see, for example, \cite{MR1052571,MR2401624})

\subsection{Bestvina-Feighn-Handel Laminations}

Fix a free basis $\mathcal{A}$ of $F_N$.  Let $\partial F_N$ be the \emph{Gromov boundary} of the word hyperbolic group, $F_N$; that is, $\partial F_N := \{ a_1a_2 \dots \mid a_i \in \mathcal{A}^{\pm} , a_i^{-1} \neq a_{i+1} \}$.  The \emph{double boundary} of $F_N$, denoted by $\partial^2F_N$, is defined by $\partial^2F_N := (\partial F_N \times \partial F_N) \backslash \Delta$, where $\Delta \subset \partial F_N \times \partial F_N$ consists of those pairs $(\zeta_1,\zeta_2)$ for which $\zeta_1 = \zeta_2$.  More generally, given a marked graph $(\Gamma,\tau)$, one defines $\partial^2 \widetilde{\Gamma}$ by considering instead one sided infinite reduced edgepaths in $\Gamma$.

\begin{defn}[Bestvina-Feighn-Handel lamination]\label{bfhl}
Let $\varphi \in \Out(F_N)$ be a fully irreducible automorphism equipped with a train track map $f \colon \Gamma \to \Gamma$.  The \emph{Bestvina-Feighn-Handel lamination} associated to $\varphi \in \Out(F_N)$, denoted by $\mathcal{L}_{BFH}(\varphi,f,\Gamma)$, is the set of pairs $(\zeta_1, \zeta_2) \in \partial^2 \widetilde{\Gamma}$ which have the following property: for every finite subpath $[z_1,z_2] \Subset (\zeta_1,\zeta_2)$ there exists an $e \in E(\Gamma)$ and an $n \geq 1$ so that $f^n(e) \Supset \pi([z_1,z_2])$.  Here $\pi \colon \partial^2 \widetilde{\Gamma} \to \Gamma$ is the ``labelling'' map which, on each coordinate, coincides with projection from the universal covering, $\pi \colon \widetilde{\Gamma} \to \Gamma$.  Such a pair $(\zeta_1,\zeta_2)$ is called a \emph{leaf} of the lamination.
\end{defn}

If $f \colon \Gamma \to \Gamma$ is a train track representing an iwip, then given any pair of edges $e, e' \in E(\Gamma)$, there exists an $n \geq 1$ so that either $f^n(e) \Supset e'$ or $f^n(e^{-1}) \Supset e'$.  In what follows, we will always assume that we have passed to a power of $\varphi$ so that this is indeed the case.  Definition \ref{bfhl} can be formulated (equivalently) in terms of a fixed edge $e \in E(\Gamma)$.  Furthermore, the leaves of the Bestvina-Feighn-Handel lamination satisfy a certain ``recurrence'' property, which we now formulate precisely.

\begin{defn}[Quasiperiodicity]\label{qp}
A leaf $(\zeta_1,\zeta_2)$ of $\mathcal{L}_{BFH}(\varphi,f,\Gamma)$ is said to be \emph{quasiperiodic} if for every $L > 0$, there exists $L' > L$ so that the following holds: if $[z_1,z_2]$, and $[w_1,w_2]$ are subpaths of $(\zeta_1,\zeta_2)$ for which $| [z_1,z_2] | > L'$ and $| [w_1,w_2] | < L$, then $\pi([w_1,w_2]) \Subset \pi([z_1,z_2])$ (here $\pi \colon \widetilde{\Gamma} \to \Gamma$ is the labeling map).
\end{defn}

\begin{prop}[{Bestvina, Feighn, and Handel \cite[Proposition 1.8]{MR1445386}}]\label{BFH1.8}
Every leaf $(\zeta_1,\zeta_2)$ of $\mathcal{L}_{BFH}(\varphi,f,\Gamma)$ is quasiperiodic.  \qed
\end{prop}

\begin{defn}[Carrying a leaf]\label{carries}
Let $H \leq F_N$ be a finitely generated subgroup.  Let $\varphi \in \Out(F_N)$ be a fully irreducible automorphism equipped with a train track map $f \colon \Gamma \to \Gamma$ (the marking on $\Gamma$ being $\tau \colon R_N \to \Gamma$).  Let $T = (\Gamma,\tau,l)$.  Let $X^T_H$ be the Stallings subgroup graph corresponding to the (conjugacy class of the) subgroup $\tau(H) \leq \pi_1(\Gamma,\tau(\ast))$.  We say that (the conjugacy class of) $H$ \emph{carries} the leaf $(\zeta_1,\zeta_2)$ of $\mathcal{L}_{BFH}(\Phi,f,\Gamma)$ if for every finite subpath $[z_1,z_2]$ of $(\zeta_1,\zeta_2)$, the map
\[
\pi|_{[z_1,z_2]} \colon [z_1,z_2] \to \Gamma
\]
factors through $X^T_H$ as a map $[z_1,z_2] \to X^T_H \to \Gamma$.
\end{defn}

Of particular interest to us is the following proposition.

\begin{prop}[{Bestvina, Feighn, and Handel \cite[Lemma 2.4]{MR1445386}}]\label{BFH2.4}
Let $\varphi \in \Out(F_N)$ be a fully irreducible automorphism equipped with a train track $f \colon  \Gamma \to \Gamma$.  If a finitely generated subgroup $H \leq F_N$ carries a leaf of $\mathcal{L}_{BFH}(\varphi,f,\Gamma)$, then $[F_N : H] < \infty$.  \qed
\end{prop}

\subsection{Stability of Quasi-Geodesics}  We will also need some basic results about stability of quasi-geodesics.

\begin{prop}[{Bridson and Haefliger \cite[III.H Theorem 1.7]{MR1744486}}]\label{BH_1.7}
For all $\delta > 0, \lambda \geq 1, \epsilon \geq 0$ there exists a constant $R = R(\delta,\lambda,\epsilon)$ with the following property: If $X$ is a $\delta$-hyperbolic geodesic space, $c$ is a $(\lambda,\epsilon)$-quasi-geodesic in $X$ and $[p,q]$ is a geodesic segment joining the endpoints of $c$, then the Hausdorff distance between $[p,q]$ and the image of $c$ is less than $R$.
\end{prop}

\begin{prop}\label{diam}
Let $X,Y$ be $\delta$-hyperbolic geodesic metric spaces equipped with a $(\lambda,\epsilon)$-quasi-isometry $f \colon X \to Y$.  Suppose $A, B$ are collections of geodesic paths in $X$ with the following property, there exists a constant $C$ such that for any $\alpha \in A$ and $\beta \in B$, $\diam_X(\alpha \cap \beta) < C$.  Let $\alpha' = [f(\alpha)], \beta'=[f(\beta)]$ be geodesics in $Y$ obtained by reducing the images $f(\alpha), f(\beta)$, respectively.  The there exists a constant $C'$ such that for all $\alpha \in A, \beta \in B$, we have $\diam_Y(\alpha' \cap \beta') < C'$.  
\end{prop}
\begin{proof}
Let $R=R(0,\lambda,\epsilon)$ be the constant afforded by Proposition \ref{BH_1.7}.  Let $\{ \alpha_i \}, \{ \beta_i \} \subset X$ and and $\{ \alpha'_i \}, \{ \beta'_i \} \subset Y$ be arbitrary sequences of geodesics (with $\alpha'_i, \beta'_i$ contained in the $R$-neighborhood of $f(\alpha_i), f(\beta_i)$, respectively).  Suppose without loss that $l_X(\alpha_i), l_X(\beta_i) \to \infty$ (the proposition is obvious otherwise).  Since $R$ is finite and since $f$ is a quasi-isometry, we must have that $l_X(\alpha_i), l_X(\beta_i) \to \infty$.  Similarly, the distance (in $Y$) between $f(\alpha_i \cap \beta_i)$ and $\alpha'_i \cap \beta'_i$ is uniformly bounded (independent of $i$).  Suppose for contradiction that no such $C'$ exists; that is, $\diam_Y(\alpha'_i \cap \beta'_i) \to \infty$.  Then (by passing to a subsequence and reindexing) we may find a sequence of points $a_i \in \alpha_i, b_i \in \beta_i$ with $d_X(a_i,\alpha_i \cap \beta_i), d_X(b_i,\alpha_i \cap \beta_i) \to \infty$ for which $d_Y(f(a_i),f(b_i))$ is uniformly bounded.  This contradicts the fact that $f$ is a quasi-isometry.
\end{proof}

\begin{cor}\label{qi}
Let $(\Gamma,\tau)$ be a marked graph.  Suppose $z$ is a loop in $G$ which represents a primitive element of $\pi_1(\Gamma)$.  Let $A$ be an arbitrary collection of loops in $G$.  Suppose there exists a number $M$ such that for all $a \in A$ whenever $z^k \Subset a$ it follows that $|k| \leq M$.  Then for any free basis $\mathcal{B}$ with $z \in \mathcal{B}$ and a number $M'$ such that for all $a \in A$ whenever $z^k \Subset [a]_\mathcal{B}$ it follows that $|k| \leq M'$.
\end{cor}
\begin{proof}
Let $Z = \{ z^k \mid k \in \mathbb{Z} \}$.  Let $f \colon \widetilde{\Gamma} \to \Cay(F_N,\mathcal{B})$ be a quasi-isometry with $\mathcal{B}$ any free basis containing $z$.  We may think of $Z$, $A$ as defining a collection of geodesics in $\widetilde{\Gamma}$.  The assumption on $z$ affords a constant, $C$, such that for any $\zeta \in Z$ and any $\alpha \in A$, we have $\diam_{\widetilde{\Gamma}}(\zeta \cap \alpha) < C$.  We now apply Proposition \ref{diam} and conclude that (in $\Cay(F_N,\mathcal{B})$) there is a bound, $C'$ on $\diam_{\Cay(F_N,\mathcal{B})}(\zeta',\alpha')$.  Thus there is a bound $M'$ so that if $z^k \Subset [a]_\mathcal{B}$, then $|k| \leq M'$.
\end{proof}

\subsection{Geodesic Currents}\label{GC}

A \emph{geodesic current} on $F_N$ is a positive Radon measure (a Borel measure that is finite on compact sets) on $\partial^2 F_N$ which is invariant under the diagonal action of $F_N$ and under the map which interchanges the coordinates of $\partial^2 F_N$.  The space of all currents is denoted $\Curr(F_N)$.  If $\nu, \mu \in \Curr(F_N)$ are (nontrivial) currents such that $\nu = \lambda \mu$ for some $\lambda \in \mathbb{R}^*$ then we write $[\nu] = [\mu]$, where $[\cdot]$ denotes the projective class of a given current.  To each root free conjugacy class $g \in F_N$ we associate a \emph{counting curent}, $\eta_g$ as follows.  Let $\mathcal{A}$ be a free basis of $F_N$.  Write $g$ as a cyclically reduced word over $\mathcal{A}$.  We can thus think of $g$ as the label of a directed graph, $\Gamma_g$, which is topologically homeomorphic to $\mathbb{S}^1$ and has $\|g\|_\mathcal{A}$ edges (and hence $\| g \|_\mathcal{A}$ vertices), each edge labelled by the appropriate $a_i \in \mathcal{A}$.  Now let $v$ be any freely reduced word.  By the number of occurrences of $v$ or $v^{-1}$ in $g$, denoted $n_g(v^\pm)$, we mean the number of vertices of $\Gamma_g$ at which one may read the word $v$ or $v^{-1}$ along $\Gamma_g$ (going `with' the oriented edges) without leaving $\Gamma_g$.  Note that $n_u(v^\pm) = n_{u^{-1}}(v^\pm)$ by definition.  Now let $\widetilde{v}$ be a lift of $v$ to the Cayley tree of $F_N$ with respect to $\mathcal{A}$.  Let $\Cyl(v) \subset \partial^2 F_N$ be the set of all $(\zeta_1,\zeta_2) \subset \partial^2(F_N)$ such that the bi-infinite geodesic representing $(\zeta_1,\zeta_2)$ passes through $\widetilde{v}$.  Then $\eta_g(Cyl(v)) = n_g(v^\pm)$.  If $g = h^k$ for $h$ a root free conjugacy class, we define $n_g(v^\pm) = k n_h(v^\pm)$.  Currents of the form $\lambda \eta_g$ for $\lambda > 0$ and $g \in F_N$ form a dense subset of $\Curr(F_N)$ \cite[Corollary 3.5]{MR2197815}.  For more information, see \cite{MR2216713,MR2496058,MR2585579,Martin}.  We will need the notion of an intersection form $\langle \cdot , \cdot \rangle \colon \overline{\cv}_N \times \Curr(F_N) \to \mathbb{R}$.

\begin{prop}[{Kapovich and Lustig \cite[Theorem A]{MR2496058}}]\label{KLA}
There exists a map $\langle \cdot , \cdot \rangle \colon \overline{\cv}_N \times \Curr(F_N) \to \mathbb{R}$ which is continuous, $\Out(F_N)$-invariant, and linear with respect to the second argument.  Furthermore for every $g \in F_N$, and $T \in \overline{\cv}_N$ we have $\langle T , \eta_g \rangle = \| g \|_T$.
\end{prop}

\section{proof of theorem a}

In what follows, whenever we speak of an iwip $\varphi \in \Out(F_N)$ equipped with train track map $f \colon \Gamma \to \Gamma$ we assume that we have passed to a suitable power of $\varphi$ so that for any topological edge $e$ in $\Gamma$, $f(e)$ crosses each topological edge in $\Gamma$.  We will, however, continue to write $\varphi, f$.

\begin{prop}\label{1}
Let $H \leq F_N$ be a finitely generated subgroup of infinite index.  Let $\varphi \in \Out(F_N)$ be fully irreducible with train track map $f \colon \Gamma \to \Gamma$ on the marked graph $(\Gamma,\tau)$.  Let $T = (\Gamma,\tau,l)$ in $\cv_N$.  Let $X^T_H$ be the Stallings subgroup graph corresponding to the conjugacy class of $\tau(H)$ in $\pi_1(\Gamma)$.  Then there exists a power $m$, such that for all $n \geq m$, and for any $e \in E(\Gamma)$, the edgepath $f^n(e)$ does not factor through $X^T_H$.  Furthermore, given a basepoint $b$ in $(X^T_H)_r$, we may choose $m$ so that for all $n \geq m$, $f^n(e)$ does not factor through $(X^T_H)_b$.
\end{prop}
\begin{proof}
By Proposition \ref{BFH2.4}, $H$ can not carry a leaf of $\mathcal{L}_{BFH}(\varphi,f,\Gamma)$.  Thus there is a finite subpath $[z_1,z_2]$ (of a leaf $(\zeta_1,\zeta_2) \subset \partial^2 \widetilde{\Gamma}$) which does not factor through $X^T_H$.  Let $L = | [z_1,z_2] |$.  Since leaves are quasiperiodic, there exists $L'$ so that for any subpath $[w_1,w_2]$ of length at least $L'$, we have $[w_1,w_2] \Supset [z_1,z_2]$.  Choose a basepoint $v \in \Gamma$ and let $b$ be the basepoint in $(X^T_H)_r$ so that the image of $\pi_1((X^T_H)_b)$ in $\pi_1(\Gamma,v)$ is equal to $\tau(H)$.  Let $l$ be the length of the bridge $[X^T_H,b]$.  Set $L'' = L' + 2l + 2 \diam(X^T_H) + 1$.  Now let $m \geq 1$ be so that the length of $f^m(e)$ is at least $L''$ for any $e \in E(\Gamma)$.  Then for any $n \geq m$, $f^n(e) \Supset [z_1,z_2]$ for any $e \in E(\Gamma)$.  Thus for any $n \geq m$, and for any $e \in E(\Gamma)$, we have that $f^n(e)$ does not factor through $X^T_H$, nor can it factor through $(X^T_H)_b$ since at least $2l + 1$ edges in $f^n(e)$ must lie outside of $X^T_H$.
\end{proof}

\begin{cor}\label{2}
Let $\mathcal{H} = \cup H_i$ be a finite union of subgroups $H_i < F_N$ of infinite index.  Then there exists a power $m$ such that for all $n \geq m$ and for any $e \in E(\Gamma)$ and for all $i$, the edgepath $f^n(e)$ does not factor through $(X^T_{H_i})_b$.  Furthermore, there exists a primitive element $z \in \pi_1(\Gamma)$ so that $z$ contains $f^m(e)$ and hence $z$ does not factor through $(X^T_{H_i})_b$ for any $i$.
\end{cor}
\begin{proof}
Fix a fully irreducible $\varphi \in \Out(F_N)$ equipped with train track map $f \colon \Gamma \to \Gamma$.  Fix an edge $e \in E(\Gamma)$.  Now apply Proposition \ref{1} to each $H_i$ and obtain a collection $\{ m_i \}_{i=1}^k$ of exponents for which the following holds for each $i$: for all $n \geq m_i$ the edgepath $f^n(e)$ does not factor through $(X^T_{H_i})_b$.  Set $m = \max_i \{ m_i \}$.  Then for all $n \geq m$ and for all $i$ we have that $f^n(e)$ does not factor through $(X^T_{H_i})_b$.

Now let $a$ be any primitive loop in $\pi_1(\Gamma)$.  Since $\varphi$ is fully irreducible, $z$ cannot be a periodic conjugacy class, and hence $l_\Gamma([[f^k(a)]]) \to \infty$.  Thus we can find an edge $e'$ in suitable iterate, $f^k(a)$, of $a$ so that in all further iterates of $a$ by $f$ there is no cancellation in the images of $e'$.  Thus for all $n \geq k+m+1$, we have that $f^m(e)$ occurs in $f^n(a)$ and hence $f^n(a)$ cannot factor through $(X^T_{H_i})_b$ for any $i$.  Set $z = f^n(a)$ for some $n \geq k+m+1$.
\end{proof}

\begin{prop}\label{3}
Let $\mathcal{H} = \cup H_i$ be a finite union of subgroups $H_i < F_N$ of infinite index.  Then there exists a primitive element $z$ and a free basis $\mathcal{B}$ with $z \in \mathcal{B}$ and a number $M$ so that for all $i$ and for all $h \in H_i$, if $z^k \Subset [h]_\mathcal{B}$, then $|k| \leq M$.
\end{prop}
\begin{proof}
Let $z$ be the primitive element afforded by Corollary \ref{2}.  Let $Z = \{ z^k \mid k \in \mathbb{Z} \}$.  Let $B = \{ \tau(h) \mid h \in H \}$.  The previous proposition gives the following: there exists a constant $C$ so that for any $\zeta \in Z$ and for any $\beta \in B$ and for any choice of lifts $\widetilde{\zeta}, \widetilde{\beta}$ in $\widetilde{\Gamma}$, the intersection $\diam(\widetilde{\zeta} \cap \widetilde{\beta}) < C$.  Let $f \colon \widetilde{\Gamma} \to \Cay(F_N,\mathcal{B})$ be a quasi-isometry to the Cayley graph of $F_N$ with respect to a basis $\mathcal{B}$ with $z \in \mathcal{B}$.  Now apply Proposition \ref{BH_1.7} using $Z,B$ and $f$.  We conclude there exists a bound $M$ so that for any $h \in \mathcal{H}$, if $z^k \Subset [h]_\mathcal{B}$, then $|k| \leq M$.
\end{proof}

\begin{thm}\label{vW}
Let $\mathcal{H}= \cup g_iH_i$ be a finite collection of cosets of finitely generated subgroups of infinite index, $H_i \leq F_N$.  Then the set $\mathcal{H}$ has property $\mathcal{W}$.
\end{thm}
\begin{proof}
Let $z, \mathcal{B}$ be the primitive element and free basis afforded by Proposition \ref{3}.  Thus there exists an $M'$ such that for all $i$ and all $h \in H_i$ if $z^k \Subset [h_i]_\mathcal{B}$, then $|k| \leq M'$.  Since $\{ g_i \}_{i=1}^n$ is finite, there is $M_g$ so that for all $i$, if $z^k \Subset [g_i]_\mathcal{B}$, then $|k| \leq M_g$.  Thus there is an $M$ so that for any $i$ and any $h_i \in H_i$, if $z^k \Subset [[g_i h_i]]_\mathcal{B}$, then $|k| \leq M$.  Thus $H$ has property $\mathcal{W}$ with respect to $(z,\mathcal{B},M)$.
\end{proof}

Since $\| g^n \| = n \| g \|$ it is clear that any (coset of a) subgroup of finite index is spectrally rigid.  We thus have the following.

\begin{thmaa}
Let $N \geq 2$.  Let $\{ H_i \}_{i=1}^k$ be a finite collection of finitely generated subgroups $H_i < F_N$.  Let $\{ g_i \}_{i=1}^k$ be a finite collection of elements $g_i \in F_N$. Let $\mathcal{H} = \cup_{i=1}^k g_iH_i$.  Then the following are equivalent.
\begin{enumerate}
\item For every $i$, $[F_N \colon H_i] = \infty$.
\item $\mathcal{H} \subset F_N$ satisfies property $\mathcal{W}$.
\item $\mathcal{H} \subset F_N$ is not spectrally rigid in $F_N$.
\end{enumerate}
\end{thmaa}
\begin{proof}
Theorem \ref{vW} gives (1) implies (2).  Theorem \ref{CLSRSV} gives (2) implies (3).  Finally, if at least one subgroup is of finite index (i.e. not (1)), then $\mathcal{H}$ is spectrally rigid (i.e. not (3)).
\end{proof}

\section{proof of theorem b}

Recall that the space, $\Curr(F_N)$, of geodesic currents consists of positive radon measures on $\partial^2 F_N$ which are $F_N$ and flip invariant.  Also recall that $\Curr(F_N)$ and $\overline{\cv}_N$ are related by an intersection form $\langle \cdot , \cdot \rangle$ for which $\langle \eta_g , T \rangle = \| g \|_T$ (see \S \ref{GC} and Proposition \ref{KLA}).

\begin{prop}[{Kapovich and Lustig \cite[Lemma 2.10]{MR2322181}}]\label{KL2.10}
Let $\mathcal{A}$ be a free basis of $F_N$.  Then for any cyclically reduced words $w_n, w \in F_N$, we have
\[
\lim_{n \to \infty} [\eta_{w_n}] = [\eta_w] \qquad \text{if and only if} \qquad \lim_{n \to \infty} \frac{\eta_{w_n}}{\| w_n \|_\mathcal{A}} = \frac{\eta_w}{\| w \|_\mathcal{A}}
\]
\end{prop}

\begin{prop}
For any nontrivial $r \in F_N$, the set $\ncl(r)$ is strongly spectrally rigid.
\end{prop}
\begin{proof}
Let $u \in F_N$ be an arbitrary cyclically reduced word.  We will show that $\| u \|_T = \| u \|_{T'}$ from which it will follow that $T=T'$ (since $\| \cdot \|$ is constant on conjugacy classes).  To that end, we first construct a sequence of cyclically reduced words $\{w_i\} \subset \ncl(r)$ with the property that
\[
[\eta_{w_i}] \to [\eta_{u}]
\]
To that end, let
\[
w_i = \alpha u^i \beta r \beta^{-1} u^{-i} \alpha^{-1} \gamma u^i \delta r \delta^{-1} u^{-i} \gamma^{-1}
\]
Evidently $w_i$ is a product of conjugates of $r$, so is in $\ncl(r)$.  It is clear that we may choose $\alpha, \beta, \delta,\gamma$ so that $w_i$ is freely (and so cyclically, by inspection) reduced.  Now let $v \in F_N$ be arbitrary.  By Proposition \ref{KL2.10} it is enough to show that
\[
\lim_{i \to \infty} \frac{n_{w_i}(v^\pm)}{\| w_i \|_\mathcal{A}} = \frac{n_u(v^\pm)}{\| u \|_\mathcal{A}}
\]
Note that there is a uniform bound on the number of occurrences of $v$ or $v^{-1}$ in $w_i$ which do not occur entirely within the $u^i$ or $u^{-i}$ subwords of $w_i$.  Furthermore the difference between $\| w_i \|_\mathcal{A}$ and $4 \| u^i \|_\mathcal{A}$ is uniformly bounded.  We now have that
\begin{align*}
\lim_{i \to \infty} \frac{n_{w_i}(v^\pm)}{\| w_i \|_\mathcal{A}} &= \lim_{i \to \infty} \frac{2n_{u^i}(v^\pm) + 2n_{u^{-i}}(v^\pm) + D_i}{4i \| u \|_\mathcal{A} + C_i} \\
&= \lim_{i \to \infty} \frac{2in_{u}(v^\pm) + 2in_{u^{-1}}(v^\pm) + D_i}{4i \| u \|_\mathcal{A} + C_i} \\
&= \lim_{i \to \infty} \frac{2in_{u}(v^\pm) + 2in_{u}(v^\pm) + D_i}{4i \| u \|_\mathcal{A} + C_i} \\
&= \lim_{i \to \infty} \frac{4in_{u}(v^\pm) + D_i}{4i \| u \|_\mathcal{A} + C_i}= \frac{n_u(v^\pm)}{\| u \|_\mathcal{A}}
\end{align*}
as desired.

Now suppose that $T,T' \in \overline{\cv}_N$ agree on $\ncl(r)$.  Since $[\eta_{w_i}] \to [\eta_u]$ there exists a sequence $\lambda_i$ such that
\[
\lim_{i \to \infty} \lambda_i \eta_{w_i} = \eta_u
\]
Then we have that
\begin{align*}
\| u \|_T &= \langle T , \eta_u \rangle = \langle T , \lim_{i \to \infty} \lambda_i \eta_{w_i} \rangle = \lim_{i \to \infty} \lambda_i  \langle T , \eta_{w_i} \rangle \\
&= \lim_{i \to \infty} \lambda_i \| w_i \|_T = \lim_{i \to \infty} \lambda_i \| w_i \|_{T'} = \lim_{i \to \infty} \lambda_i \langle T' , \eta_{w_i} \rangle \\
&= \langle T' , \lim_{i \to \infty} \lambda_i \eta_{w_i} \rangle = \langle T' , \eta_u \rangle = \| u \|_{T'}
\end{align*}
Recall that $u$ was an arbitrary conjugacy class.  Since length functions are class functions, we conclude that $\| \cdot \|_T = \| \cdot \|_{T'}$ and so $T = T'$ in $\overline{\cv}_N$.
\end{proof}

\begin{cor}
Let $r \in F_N$ be nontrivial and consider $g \ncl(r)$.  Then $g \ncl(r)$ is strongly spectrally rigid.
\end{cor}
\begin{proof}
In the proof of the above proposition, we may choose $\alpha,\beta,\gamma,\delta$ so that $g w_i$ is cyclically reduced.  Then -- with perhaps different constants $C_i, D_i$ -- we obtain that $[\eta_{g w_i}] \to [\eta_u]$.  We conclude as above that $g \ncl(r)$ is strongly spectrally rigid.
\end{proof}

\begin{thmaa}
Let $H \lhd F_N$ be a nontrivial normal subgroup.  Then for any $g \in F_N$, the coset $gH$ is strongly spectrally rigid.
\end{thmaa}
\begin{proof}
Since $H$ is nontrivial and normal, $H \supset \ncl(r)$ for some nontrivial $r$.  By the above proposition, for any $g \in F_N$, $g \ncl(r)$ is strongly spectrally rigid.  Thus so is $gH$.
\end{proof}


\end{document}